\newfont{\Bbb}{msbm10 scaled\magstephalf}
\documentclass[reqno]{amsart}
\usepackage{amsfonts}
\usepackage{amsmath, amsthm, amscd, amsfonts}
\usepackage{amssymb}
 \newtheorem{thm}{Theorem}[section]
 \newtheorem{cor}[thm]{Corollary}
 \newtheorem{Lemma}[thm]{Lemma}
 
 \theoremstyle{definition}
 
\theoremstyle{remark}

 \numberwithin{equation}{section}

\begin{document}
\title[Weighted composition operator]{Estimates of essential norms of weighted composition operator from Bloch type spaces to
Zygmund type spaces}
\author[Y.X. Liang and Z.H. Zhou]{Yu-Xia Liang and Ze-Hua Zhou$^*$}

\address{\newline  Yu-Xia Liang\newline Department of Mathematics\newline
Tianjin University\newline Tianjin 300072, P.R. China.}
\email{liangyx1986@126.com}

\address{\newline  Ze-Hua Zhou\newline Department of Mathematics\newline Tianjin University\newline Tianjin
300072, P.R. China.} \email{zehuazhoumath@aliyun.com;zhzhou@tju.edu.cn}



\date{}
\thanks{\noindent $^*$Corresponding author.\\
This work was supported in part by the National Natural Science Foundation of
China (Grant Nos. 11371276; 11301373; 11201331).}

\begin{abstract}  Let $u$ be a holomorphic function and $\varphi$ a holomorphic self-map of the open unit disk $\mathbb{D}$ in the complex plane. We give some new characterizations for the boundedness of the weighted composition operators $uC_{\varphi}$ from Bloch type spaces to
Zygmund type spaces in $\mathbb{D}$ in terms of $u, \varphi$, their derivatives and the $n$-th
power $\varphi^n$ of $\varphi$. Moreover, we obtain some similar estimates for their essential norms. From which the sufficient and necessary conditions of compactness of the operators $uC_{\varphi}$ follows immediately.
\end{abstract}

\maketitle
\section{Introduction}
 Let $\mathbb{D}$ be the open unit disk in the complex plane $\mathbb{C}$. Let $H(\mathbb{D})$ denote the class
of all functions analytic on $\mathbb{D}$ and $S(\mathbb{D})$ the collections
of all holomorphic self-map of $\mathbb{D}$. We give the weighted Banach spaces of analytic functions
$$H_\nu^\infty=\{f\in H(\mathbb{D}): \;\|f\|_\nu:=\sup\limits_{z\in\mathbb{D}} \nu (z)|f(z)|<\infty\}$$
endowed with norm $\|.\|_\nu$, where the weight $\nu:\mathbb{D}\rightarrow \mathbb{R}_+$ is a continuous strictly positive and bounded function. The
weight $\nu$ is called radial, if $\nu(z)=\nu(|z|)$ for all $z\in \mathbb{D}$. For a weight $\nu$ the associated weight $\tilde{\nu}$ is defined by $$\tilde{\nu}(z):=(\sup\{|f(z)|; \;f\in H_\nu^\infty,\; \|f\|_\nu \leq 1\})^{-1},\;z\in\mathbb{D}.$$
It is obvious that $\tilde{\nu}_\alpha=\nu_\alpha$ for  the standard weights $\nu_\alpha(z)=(1-|z|^2)^\alpha,$ where $0<\alpha<\infty$. Besides the standard weights $\nu_\alpha$, we also consider the  logarithmic weight  $$\nu_{log}(z):=\left(\log \left( \frac{e}{1-|z|^2}\right)\right)^{-1},\;\;z\in\mathbb{D}.$$ It is not difficult to see that also $\tilde{\nu}_{\log}= \nu_{\log}$. Moreover, the Banach space of bounded analytic functions on $\mathbb{D}$ is denoted by $H^\infty.$ In the following,  let $\|f\|_{\nu_\alpha}$ and $\|f\|_{\nu_{\log}}$ denote the norms  defined on the weighted Banach spaces $H_{\nu_\alpha}^\infty$ and $H_{\nu_{\log}}^\infty$.

Recall that the Bloch type space $\mathcal{B}^\alpha$ on the unit disk, consists of all $f\in H(\mathbb{D})$ satisfying
\begin{eqnarray*}\|f\|_{\alpha}:=\sup\limits_{z\in \mathbb{D}}(1-|z|^2)^\alpha|f'(z)|<\infty\end{eqnarray*}  endowed with the norm
$\|f\|_{\mathcal {B}_\alpha}=|f(0)|+\|f\|_\alpha<\infty.$
As we all known that for $0<\alpha<1$, $\mathcal{B}^\alpha$ is a subspace of $H^\infty.$ When $\alpha=1,$ we get the classical Bloch space $\mathcal{B}.$

For $0<\beta<\infty$, we denote by $\mathcal{Z}_\beta$ the Zygmund type space of those functions $f\in H(\mathbb{D})$ such that
$$ \sup\limits_{z\in\mathbb{D}} (1-|z|^2)^\beta |f''(z)|<\infty$$ equipped with the norm
\begin{eqnarray*}\|f\|_{\mathcal{Z}_\beta}:=|f(0)|+|f'(0)|+\sup\limits_{z\in\mathbb{D}} (1-|z|^2)^\beta |f''(z)|. \end{eqnarray*}
For $\beta=1$ we obtain the classical Zygmund space $\mathcal{Z}$.

The composition operator $C_\varphi$ induced by $\varphi\in S(\mathbb{D})$ is defined on $H(\mathbb{D})$  by $C_\varphi(f) = f\circ\varphi$ for any  $f\in H(\mathbb{D})$. This operator is well studied for many years, we refer to the books \cite{CM, Sha}, which are excellent
sources for the development of the theory of composition operators in function spaces.
For $u\in H(\mathbb{D})$, we define the weighted composition operator $$uC_\varphi f(z)=u(z)f(\varphi(z)),\;\mbox{for}\;f\in H(\mathbb{D}).$$ It is obvious that $uC_\varphi= C_\varphi$ when $u$ is the identity map.


 The essential norm of a continuous linear operator $T$ is the distance from  $T$ to the compact operators $K$,
that is $\|T\|_{e}=\inf\{\|T-K\|: \;K $ is compact $\}.$  Notice that $\|T\|_{e}=0$ if and only if $T$ is compact, so estimates on
$\|T\|_{e}$  lead to conditions for $T$ to be compact. There are lots of papers concerning this topic, the interested readers can refer to  \cite{FZ,GM,MZ,S,ZS,ZZ} and the references therein.

Recently, there is an increase interest to characterize the boundedness and compactness of composition operators acting on Bloch type spaces in terms of the $n$-th power $\varphi^n$ of $\varphi$, see \cite{LZ1,LZ2,WZZ,Zhao}. The similar characterization between
Bloch-type spaces with general radial weights was obtained by Hyv\"{a}rinen et al in \cite{HKLRS}. The natural question to ask is whether the essential norm formula for
composition operators between Bloch-types paces $\mathcal{B}_{\alpha}$ can be generalized to weighted composition operators. In 2012, Manhas and Zhao \cite{MZ1} showed that the question has an affirmative answer when $\alpha\neq 1$; however, they were not able to estimate the essential norm of weighted composition operators on the Bloch space $\mathcal{B}$. The open problem was solved by Hyv\"{a}rinen and Lindstr\"{o}m in \cite{HL}. Moreover, they presented a direct method to calculate the essential norms of weighted composition operators $uC_{\varphi}$ actingon all Bloch-type spaces $\mathcal{B}_{\alpha}$ in terms of $u$ and the $n$-th power of $\varphi$. After that, Esmaeili and Lindstr\"{o}m \cite{EL} gave similar characterizations for the weighted composition operators acting on Zygmund type spaces.

Based on the above foundations, in this paper we used an approach due to Hyv\"{a}rinen and Lindstr\"{o}m in \cite{HL} and Esmaeili and Lindstr\"{o}m in \cite{EL} to obtain new characterizations for bounded weighted composition operators from Bloch type spaces to
Zygmund type spaces, and to give similar estimates of the essential norms of such operators.

Throughout the remainder of this paper, $C$ will denote a positive constant, the exact value of which will vary from one appearance to the next.
The notation $A\preceq B$, $A\succeq B$ and $A\asymp B$ mean  that there may be different  positive constants $C$ such that $A\leq CB$, $A\geq CB$ and $B/C \leq A \leq CB$.

\section{Some Lemmas}
In this section, we give some auxiliary results which will be used in proving the main results of the paper. The following two lemmas is crucial to the new characterizations.
\begin{Lemma}\cite[Theorem 2.1]{M} or \cite[Theorem 2.4]{HKLRS} Let $\nu$ and $w$ be radial, non-increasing weights tending to zero at the boundary of $\mathbb{D}$. Then

(i) the weighted composition operator $uC_\varphi$ maps $H_\nu^\infty$ into $H_w^\infty$ if and only if
\begin{eqnarray*}\sup\limits_{n\geq0}\frac{\|u \varphi^n\|_w}{\|z^n\|_\nu} \asymp \sup\limits_{z\in \mathbb{D}} \frac{w(z)}{\tilde{ \nu}(\varphi(z))}|u(z)|<\infty, \end{eqnarray*} with norm comparable to the above supremum.\vspace{3mm}

(ii) $\|uC_\varphi\|_{e,H_\nu^\infty\rightarrow H_w^\infty}=\limsup\limits_{n\rightarrow \infty}\frac{\|u\varphi^n\|_w }{\|z^n\|_\nu}=\limsup\limits_{|\varphi(z)|\rightarrow 1}\frac{w(z)}{\bar{\nu}(\varphi(z))} |u(z)|.$
\end{Lemma}

\begin{Lemma}\cite[Lemma 2.1]{HL} For $0<\alpha<\infty$ we have

(i) $\lim\limits_{n\rightarrow \infty} (n+1)^\alpha \|z^n\|_{\nu_\alpha}=\left(\frac{2\alpha}{e}\right)^\alpha.$

(ii) $\lim\limits_{n\rightarrow \infty} (\log n)\|z^n\|_{\nu_{\log}}=1.$
\end{Lemma}
The next lemma is a well-known characterization for the Bloch-type space on the unit disc, see \cite{Zhu3}.
\begin{Lemma} For $f\in H(\mathbb{D}), m\in \mathbb{N}$ and $\alpha>0,$ then
\begin{eqnarray*}f(z)\in \mathcal{B}^\alpha  \Leftrightarrow \|f\|_{\alpha}\asymp \sup\limits_{z\in \mathbb{D}}(1-|z|^2)^{\alpha+m-1}|f^{(m)}(z)|<\infty. \end{eqnarray*} \end{Lemma}

Hence, when $f\in \mathcal{B}^\alpha,$ we have that \begin{eqnarray}\|f\|_{\alpha}\asymp \sup\limits_{z\in \mathbb{D}}(1-|z|)^{\alpha+m-1}|f^{(m)}(z)|<\infty.\label{2.1}\end{eqnarray}
where $f^{(m)}$ denotes the $m-$th order  derivative of $f\in H(\mathbb{D}).$
\begin{Lemma}\cite{OSZ,Zhu3} For $\alpha>0,\;f\in \mathcal{B}^\alpha,$ then we have that
\begin{eqnarray*}|f(z)|\leq C
\left\{
  \begin{array}{ll}
    \|f\|_{\mathcal{B}_\alpha}, & 0<\alpha<1; \\
    \|f\|_{\mathcal{B}_\alpha}\log\frac{e}{1-|z|^2}, &\alpha=1; \\
   \frac{1}{(1-|z|^2)^{\alpha-1}}\|f\|_{\mathcal{B}_\alpha}, & \alpha>1.
  \end{array}
\right.
\end{eqnarray*}for some $C$ independent of $f$. \end{Lemma}
The following lemma is a special case of \cite[Lemma 6]{SCZ}.
\begin{Lemma} For $0<\alpha<1$ and $\{f_k\} $ is an arbitrary bounded sequence in $\mathcal{B}^\alpha$ converging to 0 uniformly on the compact subsets of $\mathbb{D}$ as $k\rightarrow \infty,$ then we have that
$$\lim\limits_{k\rightarrow \infty} \sup\limits_{z\in \mathbb{D}} |f_k(z)|=0.$$ \end{Lemma}

The following criterion for compactness follows from an easy modification  of the Proposition 3.11 of \cite{CM}. Hence we omit the details.
\begin{Lemma}  Suppose $ X$ and $Y$ are two Banach spaces. Then the weighted composition operator $uC_\varphi: X\rightarrow Y$ is compact if whenever $\{f_k\}$ is bounded in $X$ and $f_k\rightarrow 0$ uniformly on compact subsets of $\mathbb{D}$, then $uC_\varphi f_k\rightarrow0$ in $Y$ as $k\rightarrow \infty.$
\end{Lemma}







\section{Boundedness }
In this section, we give some new characterizations for  the boundedness of  $uC_\varphi: \mathcal{B}^\alpha\rightarrow \mathcal{Z}_\beta$ in three cases.
\begin{thm} If $0<\alpha<1$, then $uC_\varphi$ maps $\mathcal{B}^\alpha$ boundedly into $\mathcal{Z}_\beta$ if and only if $u\in \mathcal{Z}_\beta$
and
\begin{eqnarray} \sup\limits_{z\in\mathbb{D}} \frac{(1-|z|^2)^\beta |2u'(z)\varphi'(z)+u(z)\varphi''(z)|}{(1-|\varphi(z)|^2)^\alpha}\asymp \sup\limits_{n\geq 0}(n+1)^\alpha \|(2u'\varphi'+u\varphi'')\varphi^n\|_{\nu_\beta}<\infty.\;\;\;\label{3.1}
 \end{eqnarray}
\begin{eqnarray} \sup\limits_{z\in\mathbb{D}}\frac{(1-|z|^2)^\beta|u(z)\varphi'(z)^2|}{(1-|\varphi(z)|^2)^{\alpha+1}}\asymp \sup\limits_{n\geq 0} \|u(\varphi')^2\varphi^n\|_{\nu_\beta}(n+1)^{\alpha+1}<\infty.\;\;\;\label{3.2}  \end{eqnarray}
\end{thm}
\begin{proof}Sufficiency. Suppose $u\in \mathcal{Z}_\beta$, (\ref{3.1}) and (\ref{3.2}) hold. Since $(uC_\varphi f)''=u'' C_\varphi f+(2u'
\varphi'+u\varphi'') C_\varphi f' +u(\varphi')^2 C_\varphi f''$, using Lemma 2.3 and Lemma 2.4,  for any $f\in \mathcal{B}^\alpha$,
\begin{eqnarray*}  & &(1-|z|^2)^\beta |(u C_\varphi f)''(z)\\
&\leq &  (1-|z|^2)^\beta| u''(z)| |f(\varphi(z))|+(1-|z|^2)^\beta |u(z)(\varphi'(z))^2| |f''(\varphi(z))|\\ & &+(1-|z|^2)^\beta |2u'(z)\varphi'(z)+u(z)\varphi''(z)| |f'(\varphi(z))| \\ &\preceq& \|u\|_{\mathcal{Z}_\beta}\|f\|_{\mathcal{B}_\alpha} + \frac{(1-|z|^2)^\beta |u(z)(\varphi'(z))^2|}{(1-|\varphi(z)|^2)^{\alpha+1}}\|f\|_{\mathcal{B}_\alpha}\\ &&+ \frac{(1-|z|^2)^\beta |2u'(z)\varphi'(z)+u(z)\varphi''(z)|}{(1-|\varphi(z)|^2)^{\alpha}}\|f\|_{\mathcal{B}_\alpha}<\infty,
  \end{eqnarray*}
and $|(uC_\varphi f)'(0)|\preceq |u'(0)|\|f\|_{\mathcal{B}_\alpha}+\frac{|u(0)\varphi'(0)|}{(1-|\varphi(0)|^2)^\alpha}\|f\|_{\mathcal{B}_\alpha},\;\; |uC_\varphi f(0)|\preceq \|f\|_{\mathcal{B}_\alpha}.$
From above it follows  the boundedness of $uC_\varphi: \mathcal{B}^\alpha\rightarrow \mathcal{Z}_\beta.$

Necessity. Suppose $uC_\varphi: \mathcal{B}^\alpha\rightarrow \mathcal{Z}_\beta$ is bounded
 for $0<\alpha<1$. Then choose the functions $f(z)=1,\;f(z)=z,\; f(z)=z^2,$  and  for a fixed $w\in \mathbb{D}$, take
\begin{eqnarray*} &&g_w(z)=\frac{1-|w|^2}{(1- \bar{w}z)^\alpha}-\frac{\alpha(1-|w|^2)^2}{(\alpha+2)(1- \bar{w}z)^{\alpha+1}}-\frac{2}{(\alpha+2)(1-|w|^2)^{\alpha-1}},\\&&
f_w(z)=\frac{1-|w|^2}{(1- \bar{w}z)^\alpha}-\frac{\alpha(1-|w|^2)^2}{(\alpha+1)(1- \bar{w}z)^{\alpha+1}}-\frac{1}{(\alpha+1)(1-|w|^2)^{\alpha-1}}, \end{eqnarray*}
  Then it follows that   $u\in \mathcal{Z}_\beta$,  and
   \begin{eqnarray*}&& \sup\limits_{z\in \mathbb{D}} (1-|z|^2)^\beta |2u'(z)\varphi'(z)+u(z)\varphi''(z)|<\infty,
   \nonumber\\&& \sup\limits_{z\in \mathbb{D}}(1-|z|^2)^\beta|u(z)\varphi'(z)^2|<\infty. \end{eqnarray*}
  \begin{eqnarray*} &&\sup\limits_{z\in \mathbb{D}} \frac{(1-|z|^2)^\beta |u(z)\varphi'(z)^2||\varphi(z)|^2}{(1-|\varphi(z)|^2)^{\alpha+1}}<\infty,\\&&
  \sup\limits_{z\in \mathbb{D}} \frac{(1-|z|^2)^\beta |2u'(z)\varphi'(z)+u(z)\varphi''(z)||\varphi(z)|}{(1-|\varphi(z)|^2)^\alpha}<\infty.
  \end{eqnarray*}
  From the above four inequalities, we obtain that
\begin{eqnarray} M_1:=\sup\limits_{z\in\mathbb{D}} \frac{(1-|z|^2)^\beta |2u'(z)\varphi'(z)+u(z)\varphi''(z)|}{(1-|\varphi(z)|^2)^\alpha}<\infty,\label{3.3}\end{eqnarray}
and
\begin{eqnarray}M_2:=\sup\limits_{z\in\mathbb{D}}\frac{(1-|z|^2)^\beta|u(z)\varphi'(z)^2|}{(1-|\varphi(z)|^2)^{\alpha+1}}<\infty.\label{3.4} \end{eqnarray}
Then employing  Lemma 2.1 (i) with $\nu=\nu_\alpha,\; w(z)=\nu_\beta$, Lemma 2.2 (i), (\ref{3.3}) and (\ref{3.4}),  it follows  that
\begin{eqnarray*}M_1 &\asymp& \sup\limits_{n\geq 0}\frac{\|(2u'\varphi'+u\varphi'')\varphi^n\|_{\nu_\beta}(n+1)^\alpha}{\|z^n\|_{\nu_\alpha}(n+1)^\alpha} \\&\asymp& \sup\limits_{n\geq 0}(n+1)^\alpha \|(2u'\varphi'+u\varphi'')\varphi^n\|_{\nu_\beta},   \end{eqnarray*}
\begin{eqnarray*}M_2 &\asymp& \sup\limits_{n\geq 0} \frac{\|u(\varphi')^2\varphi^n\|_{\nu_\beta}(n+1)^{\alpha+1}}{\|z^n\|_{\nu^{\alpha+1}}(n+1)^{\alpha+1}} \\&\asymp& \sup\limits_{n\geq 0}\|u(\varphi')^2\varphi^n\|_{\nu_\beta}(n+1)^{\alpha+1}.   \end{eqnarray*}
From the above inequality we obtain (\ref{3.1}) and (\ref{3.2}).
This completes the proof of the theorem.
 \end{proof}
\begin{thm} If $\alpha=1$, then $uC_\varphi$ maps $\mathcal{B} $ boundedly into $\mathcal{Z}_\beta$ if and only if (\ref{3.1}) and (\ref{3.2}) hold
and
\begin{eqnarray}\sup\limits_{n\geq 0} (\log n) \|u''(z) \varphi^n\|_{\nu_\beta}\asymp\sup\limits_{z\in \mathbb{D}}(1-|z|^2)^\beta |u''(z)|\log\frac{e}{1-|\varphi(z)|^2}<\infty.\;\;\;\;\label{3.5}\end{eqnarray}
\end{thm}
\begin{proof}
Sufficiency.  This part is similar to the proof in Theorem 3.1.

Necessity. Suppose that $uC_\varphi$ maps $\mathcal{B} $ boundedly into $\mathcal{Z}_\beta$.  Similar to the proof in Theorem 3.1. we can obtain (\ref{3.1}) and (\ref{3.2}) with $\alpha=1$, thus we  need only to show (\ref{3.5}). In this case, we choose the function
\begin{eqnarray*}h_w(z)=\frac{6}{a}\left(\log\frac{2}{1-\bar{w}z}\right)^2 -\frac{2}{a^2}\left( \log\frac{2}{1-\bar{w}z}\right)^3,\end{eqnarray*} where $a=\log\frac{2}{1-|w|^2},$ then by (\ref{3.1}) and (\ref{3.2}), we can easily obtain that \begin{eqnarray*}M_3:= \sup\limits_{z\in \mathbb{D}}(1-|z|^2)^\beta |u''(z)|\log\frac{e}{1-|\varphi(z)|^2}<\infty.\end{eqnarray*}
Then by Lemma 2.1 (i) with $\nu=\Big(\log \Big(\frac{e}{1-|z|^2}\Big)\Big)^{-1}=\nu_{\log},\; w(z)=\nu_\beta$, and Lemma 2.2 (ii),
\begin{eqnarray*} M_3&=&\sup\limits_{z\in \mathbb{D}} \frac{(1-|z|^2)^\beta|u''(z)|}{\left(\log\frac{e}{1-|\varphi(z)|^2}\right)^{-1}}\asymp
\sup\limits_{n\geq 0}\frac{\|u''\varphi^n\|_{\nu_\beta}(\log n)}{\|z^n\|_{\nu_{\log}}(\log n)}\\&\asymp& \sup\limits_{n\geq 0}(\log n ) \|u''\varphi^n\|_{\nu_\beta}<\infty. \end{eqnarray*}
from which (\ref{3.5}) follows. This completes the proof.
 \end{proof}

\begin{thm} If $\alpha>1$, then $uC_\varphi$ maps $\mathcal{B}^\alpha$ boundedly into $\mathcal{Z}_\beta$ if and only if (\ref{3.1}) and (\ref{3.2}) hold and \begin{eqnarray}\sup\limits_{n\geq 0}(n+1)^{\alpha-1} \|u''\varphi^n\|_{\nu_\beta} \asymp \sup\limits_{z\in\mathbb{D}} \frac{(1-|z|^2)^\beta|u''(z)|}{(1-|\varphi(z)|^2)^{\alpha-1}}<\infty. \label{3.6} \end{eqnarray}
\end{thm}
\begin{proof}Sufficiency. This part is similar to the proof in Theorem 3.1.

Necessity. Take the function
 \begin{eqnarray*}Q_w(z)=\frac{(\alpha+2)(1-|w|^2)}{\alpha(1-\bar{w}z)^\alpha}-\frac{(1-|w|^2)^2}{(1-\bar{w}z)^{\alpha+1}}, \end{eqnarray*}
Employing the necessity in Theorem 3.1 with $\alpha>1,$  we can obtain that  (\ref{3.1}) and (\ref{3.2}). Then  we can easily obtain  that
\begin{eqnarray*} \sup\limits_{z\in\mathbb{D}} \frac{(1-|z|^2)^\beta|u''(z)|}{(1-|\varphi(z)|^2)^{\alpha-1}}<\infty. \end{eqnarray*}
Similar to showing  the equivalence in (\ref{3.2}) we obtain (\ref{3.6}). This completes the proof.
\end{proof}

\section{essential norms}
In this section we estimate the essential norms of $uC_\varphi:\mathcal{B}^\alpha\rightarrow \mathcal{Z}_\beta$ in terms of $u, \varphi,$ their derivatives and $\varphi^n.$ Denote $\tilde{\mathcal{B}}^\alpha=\{f\in \mathcal{B}^\alpha:\;f(0)=0\}$. Let $D_\alpha: \mathcal{B}^\alpha\rightarrow H_{\nu_\alpha}^\infty$ and $ S_{\alpha}:\mathcal{B}^\alpha \rightarrow H_{\nu_{\alpha+1}}^\infty$ be the first-order derivative operator and the second-order derivative operator, respectively. By Lemma 2.2  we have that
\begin{eqnarray*}\|D_\alpha f\|_{H_{\nu_{\alpha}}}=\|f\|_{\mathcal{B}_\alpha}\;\;\mbox{and}\;\;\|S_\alpha f \|_{H_{\nu_{\alpha+1}}}\asymp\|f\|_{\mathcal{B}_\alpha} \;\;\mbox{for}\;\; f\in \tilde{\mathcal{B}}^\alpha.\end{eqnarray*}
Further by  $ (uC_\varphi f)''=u''C_\varphi f +(2u'\varphi' +u\varphi'') C_\varphi f +u(\varphi')^2 C_\varphi f'',$   it follows that
\begin{eqnarray}\|uC_\varphi\|_{e,\tilde{\mathcal{B}}^\alpha\rightarrow \mathcal{Z}_\beta}&\preceq& \| u''C_\varphi\|_{e,\tilde{\mathcal{B}}^\alpha\rightarrow H_{\nu_\beta}^\infty} +\|u(\varphi')^2 C_\varphi \|_{e, H_{\nu_{\alpha+1}}^\infty\rightarrow H_{\nu_{\beta}}^\infty} \nonumber\\&+& \|(2u'\varphi' +u\varphi'') C_\varphi\|_{e, H_{\nu_{\alpha}}^\infty\rightarrow H_{\nu_{\beta}}^\infty}.\label{4.1} \end{eqnarray}
For the upper bound, we only need to estimate the   right three essential norms.  It is obvious that every compact operator $T\in \mathcal{K}(\tilde{\mathcal{B}}^\alpha, \mathcal{Z}_\beta)$ can be extended to a compact operator $K\in \mathcal{K}(\mathcal{B}^\alpha, \mathcal{Z}_\beta) $. In fact, for every $f\in\mathcal{B}^\alpha$, $f-f(0)\in \tilde{\mathcal{B}}^\alpha,$ and we can define $K(f):=T(f-f(0))+f(0)$, which is a compact operator from $\mathcal{B}^\alpha$ to $\mathcal{Z}_\beta$, due to $K(f_k)$ has convergent subsequence when $\{f_k\}$ is a bounded sequence. In the following lemma we consider the compact operator $K_r$ on the space $\mathcal{B}^\alpha$ defined by $K_rf(z)=f(rz).$

\begin{Lemma}If $0<\alpha<\infty$ and $uC_\varphi$ is a bounded weighted composition operator from $\mathcal{B}^\alpha$ into $ \mathcal{Z}_\beta$, then \begin{eqnarray*} \|uC_\varphi\|_{e,\tilde{\mathcal{B}}^\alpha\rightarrow \mathcal{Z}_\beta}=\|uC_\varphi\|_{e, \mathcal{B}^\alpha\rightarrow \mathcal{Z}_\beta}.\end{eqnarray*} \end{Lemma}
\begin{proof} Although the proof is similar to \cite[Lemma 3.1]{EL}, we give the process for the convenience of the readers.  It is obvious that $ \|uC_\varphi\|_{e,\tilde{\mathcal{B}}^\alpha\rightarrow \mathcal{Z}_\beta}\leq \|uC_\varphi\|_{e, \mathcal{B}^\alpha\rightarrow \mathcal{Z}_\beta}.$  For the converse, let $T \in K(\mathcal{B}^\alpha, \mathcal{Z}_\beta)$ be given. Choose  an increasing sequence $(r_n)$ in $(0,1)$ converging to 1. Denote $\mathcal{A}$ the closed subspace of $\mathcal{B}^\alpha$ consists of  all constant functions. Then we have
\begin{eqnarray*}&& \|uC_\varphi -T\|_{\mathcal{B}^\alpha \rightarrow \mathcal{Z}_\beta}=\sup\limits_{\|f\|_{\mathcal{B}^\alpha} \leq 1} \|uC_\varphi (f) -T(f)\|_{\mathcal{Z}_\beta}\\ && \leq \sup\limits_{\|f\|_{\mathcal{B}_\alpha}\leq 1} \|uC_\varphi (f-f(0)) -T|_{\tilde{\mathcal{B}}^\alpha}(f-f(0))\|_{\mathcal{Z}_\beta} +\sup\limits_{\|f\|_{\mathcal{B}_\alpha} \leq 1}\|uC_\varphi (f(0))-T(f(0))\|_{\mathcal{Z}_\beta}\nonumber\\&&\leq \sup\limits_{g\in \tilde{\mathcal{B}}^{\alpha}} \|uC_\varphi (g ) -T|_{\tilde{\mathcal{B}}^\alpha}(g)\|_{\mathcal{Z}_\beta} +\sup\limits_{h\in \mathcal{A}} \|uC_\varphi (h ) -T|_{\mathcal{A}}(h)\|_{\mathcal{Z}_\beta}. \end{eqnarray*}
Hence \begin{eqnarray*}\inf\limits_{T\in \mathcal{K}(\mathcal{B}^\alpha, \mathcal{Z}_{\beta})}\|uC_\varphi-T\|_{\mathcal{B}^\alpha\rightarrow \mathcal{Z}_{\beta}} &\leq& \inf_{T\in \mathcal{K}(\mathcal{B}^\alpha, \mathcal{Z}_{\beta}) }\|uC_\varphi-T|_{\tilde{\mathcal{B}}^\alpha}\|_{\tilde{\mathcal{B}}^\alpha\rightarrow \mathcal{Z}_{\beta}}\\& +& \inf_{T\in \mathcal{K}(\mathcal{B}^\alpha, \mathcal{Z}_{\beta}) }\|uC_\varphi-T|_{ \mathcal{A} }\|_{ \mathcal{A} \rightarrow \mathcal{Z}_{\beta}} \nonumber\\& \leq&  \|u C_\varphi\|_{e,\tilde{\mathcal{B}}^\alpha\rightarrow \mathcal{Z}_\beta} +\lim\limits_{n\rightarrow \infty} \|uC_\varphi (I-K_{r_n})\|_{ \mathcal{A} \rightarrow \mathcal{Z}_{\beta}}. \end{eqnarray*}
Since   $uC_\varphi: \mathcal{B}^\alpha\rightarrow \mathcal{Z}_\beta$ is bounded, it follows that   \begin{eqnarray*} \lim\limits_{n\rightarrow \infty} \|uC_\varphi (I-K_{r_n})\|_{ \mathcal{A} \rightarrow \mathcal{Z}_{\beta}}\leq C \lim\limits_{n\rightarrow \infty} \| I-K_{r_n} \|_{ \mathcal{A} \rightarrow \mathcal{Z}_{\beta}}=0. \end{eqnarray*}Thus  we obtain  $ \|uC_\varphi\|_{e,\tilde{\mathcal{B}}^\alpha\rightarrow \mathcal{Z}_\beta}\geq \|uC_\varphi\|_{e, \mathcal{B}^\alpha\rightarrow \mathcal{Z}_\beta}. $  The proof is finished.
 \end{proof}
Since $uC_\varphi:\mathcal{B}^\alpha\rightarrow \mathcal{Z}_\beta$ is bounded, then $u''C_\varphi$ maps $\mathcal{B}^\alpha $ boundedly in $H_{\nu_\beta}^\infty$ from  $u\in \mathcal{Z}_\beta$ for $0<\alpha<1,$ (\ref{3.5}) for $\alpha=1$ and  (\ref{3.6})  for $\alpha>1$.  Then from Lemma 4.1 we can get $\| u''C_\varphi\|_{e,\tilde{\mathcal{B}}^\alpha\rightarrow H_{\nu_\beta}^\infty}=\| u''C_\varphi\|_{e, \mathcal{B}^\alpha\rightarrow H_{\nu_\beta}^\infty} $. By (\ref{4.1}) we have that
\begin{eqnarray}\|uC_\varphi\|_{e, \mathcal{B}^\alpha\rightarrow \mathcal{Z}_\beta}&\preceq& \| u''C_\varphi\|_{e, \mathcal{B}^\alpha\rightarrow H_{\nu_\beta}^\infty} +\|u(\varphi')^2 C_\varphi \|_{e, H_{\nu_{\alpha+1}}^\infty\rightarrow H_{\nu_{\beta}}^\infty} \nonumber\\&+& \|(2u'\varphi' +u\varphi'') C_\varphi\|_{e, H_{\nu_{\alpha}}^\infty\rightarrow H_{\nu_{\beta}}^\infty}.\label{4.2} \end{eqnarray}
In next lemma we give the estimates for the essential norm for   $uC_\varphi:\mathcal{B}^\alpha\rightarrow H_{\nu_\beta}^\infty$.
\begin{Lemma} Let $0<\alpha<\infty$, the weighted composition operator $uC_\varphi:\mathcal{B}^\alpha\rightarrow H_{\nu_\beta}^\infty$ be bounded.

(i) If $0<\alpha<1,$ then $uC_\varphi:\mathcal{B}^\alpha\rightarrow H_{\nu_\beta}^\infty$ is compact.

(ii) If $\alpha=1$, then\begin{eqnarray*} \|uC_\varphi\|_{e,\mathcal{B}\rightarrow H_{\nu_\beta}^\infty} \asymp \limsup\limits_{n\rightarrow \infty} (\log n) \|u\varphi ^n\|_{\nu_\beta}.\end{eqnarray*}

(iii) If $\alpha>1$, then
\begin{eqnarray*}\|uC_\varphi\|_{e,\mathcal{B}^\alpha\rightarrow H_{\nu_\beta}^\infty}\asymp \limsup\limits_{n\rightarrow \infty} (n+1)^{\alpha-1}\|u\varphi^n\|_{\nu_\beta}.\end{eqnarray*}\end{Lemma}

\begin{proof}(i)  Since $uC_\varphi:\mathcal{B}^\alpha\rightarrow H_{\nu_\beta}^\infty$ is bounded. Choose $f(z)=1,$ we can obtain $u\in H_{\nu_\beta}^\infty.$  If $(f_n)$ is a bounded sequence in $\mathcal{B}^\alpha$ converging to zero uniformly on compact subsets of $\mathbb{D}$. By Lemma 2.5 we have that
\begin{eqnarray*}\|uC_\varphi(f_n)\|_{\nu_\beta}=\sup\limits_{z\in \mathbb{D}} (1-|z|^2)^\beta |u(z)||f_n(\varphi(z))|\leq \|u\|_{\nu_\beta}\sup\limits_{z\in \mathbb{D}}|f_n(z)| =0. \end{eqnarray*}
By Lemma 2.6 it follows that $uC_\varphi:\mathcal{B}_\alpha\rightarrow H_{\nu_\beta}^\infty$ is compact.

(ii) For $\alpha=1.$ By \cite[Theorem 3.4]{S} with $n=1$, we obtain that
\begin{eqnarray*}\|uC_\varphi\|_{e, \mathcal{B}\rightarrow H_{\nu_\beta}^\infty}\asymp \limsup\limits_{|\varphi(z)|\rightarrow1}(1-|z|^2)^\beta |u(z)|\log\frac{1+|\varphi(z)|}{1-|\varphi(z)|}. \end{eqnarray*}
since the function $\log\frac{1+x}{1-x}\asymp \log\frac{e}{1-x^2},\;x\in[0,1),$ we have
\begin{eqnarray*}\|uC_\varphi\|_{e, \mathcal{B}\rightarrow H_{\nu_\beta}^\infty}\asymp \limsup\limits_{|\varphi(z)|\rightarrow1}(1-|z|^2)^\beta |u(z)|\log\frac{e}{1-|\varphi(z)|^2}. \end{eqnarray*}
By Lemma 2.1 (ii) and Lemma 2.2 (ii) it follows that
\begin{eqnarray*} \|uC_\varphi\|_{e, \mathcal{B}\rightarrow H_{\nu_\beta}^\infty}\asymp\limsup\limits_{n\rightarrow\infty} \frac{(\log n) \|u\varphi^n\|_{\nu_\beta}}{\|z^n\|_{\nu_{\log}}(\log n)}\asymp \limsup\limits_{n\rightarrow\infty}(\log n) \|u\varphi^n\|_{\nu_\beta}. \end{eqnarray*}

(iii) For $\alpha>1.$ By \cite[Theorem 3.2]{S} with $n=1,$ it follows that  \begin{eqnarray*}\|uC_\varphi\|_{e, \mathcal{B}^\alpha\rightarrow H_{\nu_\beta}^\infty}\asymp \limsup\limits_{|\varphi(z)|\rightarrow1}\frac{(1-|z|^2)^\beta |u(z)|}{(1-|\varphi(z)|^2)^{\alpha-1}}. \end{eqnarray*}
Similarly by Lemma 2.1 (ii) and Lemma 2.2 (i) it follows (iii). This completes the proof.
 \end{proof}

\begin{thm} Let $0<\alpha<1$, the weighted composition operator $uC_\varphi:\mathcal{B}^\alpha\rightarrow \mathcal{Z}_\beta$ is bounded. Then  \begin{eqnarray*} \|uC_\varphi\|_{e,\mathcal{B}^\alpha\rightarrow \mathcal{Z}_\beta}&\asymp & \max \Big\{\limsup\limits_{n\rightarrow \infty} (n+1)^\alpha \|(2u'\varphi'+u\varphi'')\varphi^n\|_{\nu_\beta},\\ \nonumber && \;\;\;\;\;\;\;\;\;\;\;\;\;\;\;\limsup\limits_{n\rightarrow \infty} (n+1)^{\alpha+1}\| u(\varphi')^2 \varphi^n\|_{\nu_\beta}\Big\}.\end{eqnarray*}\end{thm}
\begin{proof} The boundedness of  $uC_\varphi:\mathcal{B}^\alpha\rightarrow \mathcal{Z}_\beta$ implies that $u'' C_\varphi:\mathcal{B}^\alpha\rightarrow H_{\nu_\beta}^\infty$,  $(2u'\varphi'+u\varphi'')C_\varphi: H_{\nu_\alpha}^\infty\rightarrow H_\beta^\infty$ and $ u(\varphi')^2C_\varphi: H_{\nu_{\alpha+1}}^\infty\rightarrow H_{\nu_\beta}^\infty $ are bounded  weighted composition operators by Theorem 3.1.

The upper estimate. From Lemma 4.2 it follows that $\|u''C_\varphi\|_{e,\mathcal{B}^\alpha\rightarrow H_{\nu_\beta}^\infty }=0.$ On the other hand, by Lemma 2.1 (i) and Lemma 2.2 (i),
\begin{eqnarray*} \|(2u'\varphi'+u\varphi'')C_\varphi \|_{e,H_{\nu_\alpha}^\infty\rightarrow H_{\nu_\beta}^\infty }&=&\limsup\limits_{n\rightarrow \infty} \frac{\| (2u'\varphi'+u\varphi'') \varphi^n\|_{\nu_\beta}}{\|z^n\|_{\nu_\alpha}}\nonumber\\&\asymp& \limsup\limits_{n\rightarrow \infty} (n+1)^\alpha\| (2u'\varphi'+u\varphi'') \varphi^n\|_{\nu_\beta}. \end{eqnarray*}
\begin{eqnarray*} \|u(\varphi')^2C_\varphi \|_{e,H_{\nu_{\alpha+1}}^\infty\rightarrow H_{\nu_\beta}^\infty}\asymp \limsup\limits_{n\rightarrow \infty} (n+1)^{\alpha+1}\| u(\varphi')^2 \varphi^n\|_{\nu_\beta}. \end{eqnarray*}
Thus by (\ref{4.2}) we obtain that\begin{eqnarray*} \|uC_\varphi\|_{e,\mathcal{B}^\alpha\rightarrow \mathcal{Z}_\beta}& \preceq&  \max \Big\{\limsup\limits_{n\rightarrow \infty} (n+1)^\alpha \|(2u'\varphi'+u\varphi'')\varphi^n\|_{\nu_\beta},\\ \nonumber && \;\;\;\;\;\;\;\;\;\;\;\;\;\;\;\limsup\limits_{n\rightarrow \infty} (n+1)^{\alpha+1}\| u(\varphi')^2 \varphi^n\|_{\nu_\beta}\Big\}. \end{eqnarray*}

The lower estimate. Let $\{z_k\}$ be a sequence in $\mathbb{D}$ such that  $|\varphi(z_k)|\rightarrow 1$ as $n\rightarrow \infty.$ Define
\begin{eqnarray*}f_k(z)=\frac{1-|\varphi(z_k)|^2}{(1- \overline{\varphi(z_k)}z)^\alpha}-\frac{\alpha(1-|\varphi(z_k)|^2)^2}{(\alpha+1)(1- \overline{\varphi(z_k)}z)^{\alpha+1}}-\frac{1}{(\alpha+1)(1-|\varphi(z_k)|^2)^{\alpha-1}}, \end{eqnarray*}
\begin{eqnarray*} g_k(z)=\frac{1-|\varphi(z_k)|^2}{(1-  \overline{\varphi(z_k)}z)^\alpha}-\frac{\alpha(1-|\varphi(z_k)|^2)^2}{(\alpha+2)(1- \overline{\varphi(z_k)}z)^{\alpha+1}}-\frac{2}{(\alpha+2)(1-|\varphi(z_k)|^2)^{\alpha-1}}, \end{eqnarray*}
It is obvious that $f_k$ and $g_k$ are bounded sequences in $\mathcal{B}^\alpha$ and converge to zero uniformly on compact subset of $\mathbb{D}$. By Lemma 2.6, for every compact operator $T:\mathcal{B}^\alpha\rightarrow \mathcal{Z}_\beta$, we have that $\|T f_k\|_{\mathcal{Z}_\beta}\rightarrow 0$ and $\|Tg_k\|_{\mathcal{Z}_\beta}\rightarrow 0$ as $k\rightarrow \infty$. Thus
\begin{eqnarray*}\|uC_\varphi-T\|_{\mathcal{B}^\alpha\rightarrow \mathcal{Z}_\beta}&\succeq& \limsup\limits_{k\rightarrow \infty} \|uC_\varphi (f_k)\|_{\mathcal{Z}_\beta} \nonumber\\&\succeq& \limsup\limits_{k\rightarrow \infty} \frac{\alpha(1-|z_k|^2)^\beta|u(z_k)\varphi'(z_k)^2||\varphi(z_k)|^2}{(1-|\varphi(z_k)|^2)^{\alpha+1}}\nonumber\\&\succeq &\limsup\limits_{k\rightarrow \infty} \frac{ (1-|z_k|^2)^\beta|u(z_k)\varphi'(z_k)^2| }{(1-|\varphi(z_k)|^2)^{\alpha+1}}\nonumber\\&\asymp&\limsup\limits_{n\rightarrow \infty} (n+1)^{\alpha+1}\| u(\varphi')^2 \varphi^n\|_{\nu_\beta}. \end{eqnarray*}
\begin{eqnarray*}\|u C_\varphi-T\|_{\mathcal{B}^\alpha\rightarrow \mathcal{Z}_\beta}&\succeq& \limsup\limits_{k\rightarrow \infty} \|uC_\varphi (g_k)\|_{\mathcal{Z}_\beta} \nonumber\\&\succeq& \limsup\limits_{k\rightarrow \infty} \frac{(1-|z_k|^2)^\beta |2u'(z_k)\varphi'(z_k)+u(z_k)\varphi''(z_k)|}{(1-|\varphi(z_k)|^2)^\alpha}\nonumber\\&\asymp&\limsup\limits_{n\rightarrow \infty} (n+1)^\alpha\| (2u'\varphi'+u\varphi'') \varphi^n\|_{\nu_\beta}. \end{eqnarray*}
From the above two inequalities we obtain the lower estimate. This completes the proof.
 \end{proof}
In the next two theorems, we need the following test functions sequences. Let $\{z_k\}$ be a sequence in $\mathbb{D}$ such that  $|\varphi(z_k)|\rightarrow 1$ as $k\rightarrow \infty.$ Define
\begin{eqnarray} f_k(z)=\frac{1-|\varphi(z_k)|^2}{(1-\overline{\varphi(z_k)}z)^\alpha} -\frac{(1-|\varphi(z_k)|^2)^2}{(1-\overline{\varphi(z_k)}z)^{\alpha+1}},\label{4.3}\end{eqnarray}
\begin{eqnarray}g_k(z)=\frac{1-|\varphi(z_k)|^2}{(1-\overline{\varphi(z_k)}z)^\alpha} -\frac{\alpha(1-|\varphi(z_k)|^2)^2}{(\alpha+1)(1-\overline{\varphi(z_k)}z)^{\alpha+1}},\label{4.4}\end{eqnarray}
\begin{eqnarray}h_k(z)=\frac{1-|\varphi(z_k)|^2}{(1-\overline{\varphi(z_k)}z)^\alpha} -\frac{\alpha(1-|\varphi(z_k)|^2)^2}{(\alpha+2)(1-\overline{\varphi(z_k)}z)^{\alpha+1}}.\label{4.5} \end{eqnarray}
It is easy to see that $f_k,\; g_k,\;$ and $h_k$ are all in $\mathcal{B}^\alpha$ and converge to zero uniformly on the compact subset of $\mathbb{D}$. Moreover, \begin{eqnarray*} f_k(\varphi(z_k))=0,\;\; f_k'(\varphi(z_k))=\frac{-\overline{\varphi(z_k)}}{(1-|\varphi(z_k)|^2 )^\alpha},\;\;f_k''(\varphi(z_k))=\frac{-2(\alpha+1)(\overline{\varphi(z_k)})^2}{(1-|\varphi(z_k)|^2)^{\alpha+1}}.\end{eqnarray*}
 \begin{eqnarray*} g_k(\varphi(z_k))=\frac{1}{(\alpha+1)(1-|\varphi(z_k)|^2)^{\alpha-1}},\;\; g_k'(\varphi(z_k))=0,\;\;g_k''(\varphi(z_k))=\frac{- \alpha  (\overline{\varphi(z_k)})^2}{(1-|\varphi(z_k)|^2)^{\alpha+1}}.\end{eqnarray*}
 \begin{eqnarray*} h_k(\varphi(z_k))=\frac{2}{(\alpha+2)(1-|\varphi(z_k)|^2)^{\alpha-1}},\;\; h_k'(\varphi(z_k))=\frac{ \alpha\overline{\varphi(z_k)}}{(\alpha+2)(1-|\varphi(z_k)|^2 )^\alpha},\;\;h_k''(\varphi(z_k))=0.\end{eqnarray*}

\begin{thm} Let $\alpha>1$, the weighted composition operator $uC_\varphi:\mathcal{B}^\alpha \rightarrow \mathcal{Z}_\beta$ is bounded. Then  \begin{eqnarray*} \|uC_\varphi\|_{e,\mathcal{B}^\alpha \rightarrow \mathcal{Z}_\beta}&\asymp&\max \Big\{\limsup\limits_{n\rightarrow \infty} (n+1)^\alpha \|(2u'\varphi'+u\varphi'')\varphi^n\|_{\nu_\beta},\\ \nonumber && \;\;\;\;\;\;\;\;\;\;\;\;\;\;\;\limsup\limits_{n\rightarrow \infty} (n+1)^{\alpha+1}\| u(\varphi')^2 \varphi^n\|_{\nu_\beta},\\ \nonumber && \;\;\;\;\;\;\;\;\;\;\;\;\;\;\;\;\;\limsup\limits_{n\rightarrow \infty}(n+1)^{\alpha-1} \|u''\varphi^n\|_{\nu_\beta}.\Big\}.\end{eqnarray*} \end{thm}
\begin{proof}
The upper estimate.  By Lemma 4.2 (iii) we obtain that $$ \|u''C_\varphi\|_{e,\mathcal{B}^\alpha\rightarrow H_{\nu_\beta}^\infty}\asymp \limsup\limits_{n\rightarrow \infty} (n+1)^{\alpha-1}\|u''\varphi^n\|_{\nu_\beta}.$$ Then by  (\ref{4.2}) and the proof for the upper estimate in Theorem 4.3, we obtain the upper estimate.

The lower estimate.  Let $\{z_k\}$ be a sequence in $\mathbb{D}$ such that $|\varphi(z_k)|\rightarrow 1$ as $k\rightarrow \infty.$ Using test functions defined in (\ref{4.3})-(\ref{4.5}) and for every compact operator $T:\mathcal{B}^\alpha\rightarrow \mathcal{Z}_\beta,$ it follows that
\begin{eqnarray}&&\|uC_\varphi-T\|_{\mathcal{B}^\alpha\rightarrow \mathcal{Z}_\beta}\succeq \limsup\limits_{k\rightarrow \infty} \|uC_\varphi (f_k)\|_{\mathcal{Z}_\beta}\nonumber\\&&\succeq \limsup\limits_{k\rightarrow \infty} (1-|z_k|^2)^\beta\Big| (2u'(z_k)\varphi'(z_k)+u(z_k)\varphi''(z_k)) \frac{-\overline{\varphi(z_k)}}{(1-|\varphi(z_k)|^2)^\alpha}\nonumber\\&&\;\;\;\;\;\;\;\;\;\;\;\;\;\;\;\;\;\;\;\;\;\;  +u(z_k)(\varphi'(z_k))^2\frac{-2(\alpha+1)(\overline{\varphi(z_k)})^2}{(1-|\varphi(z_k)|^2)^{\alpha+1}}  \Big|. \label{4.6}\end{eqnarray}
\begin{eqnarray}&&\|uC_\varphi-T\|_{\mathcal{B}^\alpha\rightarrow \mathcal{Z}_\beta}\succeq \limsup\limits_{k\rightarrow \infty} \|uC_\varphi (g_k)\|_{\mathcal{Z}_\beta}\nonumber\\&&\succeq \limsup\limits_{k\rightarrow \infty} (1-|z_k|^2)^\beta\Big| u''(z_k)\frac{1}{(\alpha+1)(1-|\varphi(z_k)|^2)^{\alpha-1}} \nonumber\\&&\;\;\;\;\;\;\;\;\;\;\;\;\;\;\;\;\;\;\;\;\;\;  +u(z_k)(\varphi'(z_k))^2\frac{- \alpha  (\overline{\varphi(z_k)})^2}{(1-|\varphi(z_k)|^2)^{\alpha+1}}\Big|.\label{4.7} \end{eqnarray}
\begin{eqnarray}&&\|uC_\varphi-T\|_{\mathcal{B}^\alpha\rightarrow \mathcal{Z}_\beta}\succeq \limsup\limits_{k\rightarrow \infty} \|uC_\varphi (h_k)\|_{\mathcal{Z}_\beta}\nonumber\\&&\succeq \limsup\limits_{k\rightarrow \infty} (1-|z_k|^2)^\beta\Big| u''(z_k)\frac{2}{(\alpha+2)(1-|\varphi(z_k)|^2)^{\alpha-1}} \nonumber\\&&\;\;\;\;\;\;\;\;\;\;\;\;\;\; +(2u'(z_k)\varphi'(z_k)+u(z_k)\varphi''(z_k))\frac{ \alpha\overline{\varphi(z_k)}}{(\alpha+2)(1-|\varphi(z_k)|^2 )^\alpha} \Big|.\;\;\;\label{4.8} \end{eqnarray}
Denote \begin{eqnarray*} &&A_k:=  \frac{(1-|z_k|^2)^\beta  u''(z_k) }{(1-|\varphi(z_k)|^2)^{\alpha-1}},\nonumber\\&&
B_k:= \frac{(1-|z_k|^2)^\beta (2u'(z_k)\varphi'(z_k)+u(z_k)\varphi''(z_k))\overline{\varphi(z_k)} }{(1-|\varphi(z_k)|^2)^\alpha},\nonumber\\&&C_k:= \frac{(1-|z_k|^2)^\beta  u(z_k)(\varphi'(z_k))^2 (\overline{\varphi(z_k)})^2}{(1-|\varphi(z_k)|^2)^{\alpha+1}}.
 \end{eqnarray*}
Then (\ref{4.6})-(\ref{4.8}) become
\begin{eqnarray*}&&\limsup\limits_{k\rightarrow \infty} \|uC_\varphi (f_k)\|_{\mathcal{Z}_\beta}\succeq \limsup\limits_{k\rightarrow \infty}\Big|B_k+2(\alpha+1)C_k\Big|, \nonumber\\&&\limsup\limits_{k\rightarrow \infty} \|uC_\varphi (g_k)\|_{\mathcal{Z}_\beta}\succeq \limsup\limits_{k\rightarrow \infty} \Big| \frac{A_k}{\alpha+1}-\alpha C_k\Big|,\nonumber\\
&&\limsup\limits_{k\rightarrow \infty} \|uC_\varphi (h_k)\|_{\mathcal{Z}_\beta}\succeq \limsup\limits_{k\rightarrow \infty}\Big|\frac{2A_k}{\alpha+2}+\frac{\alpha}{\alpha+2}B_k\Big| . \end{eqnarray*}

By the boundedness of $uC_\varphi:\mathcal{B}^\alpha \rightarrow \mathcal{Z}_\beta$, it follows that $\limsup\limits_{k\rightarrow \infty} \|uC_\varphi (f_k)\|_{\mathcal{Z}_\beta}<\infty,\;\limsup\limits_{k\rightarrow \infty} \|uC_\varphi (g_k)\|_{\mathcal{Z}_\beta}<\infty\;\mbox{and}\;\limsup\limits_{k\rightarrow \infty} \|uC_\varphi (h_k)\|_{\mathcal{Z}_\beta}<\infty.$ Thus we can obtain that $\limsup\limits_{k\rightarrow \infty}  |A_k|<\infty,\limsup\limits_{k\rightarrow \infty}  |B_k|<\infty\;\mbox{and}\;\limsup\limits_{k\rightarrow \infty}  |C_k|<\infty.  $   Then by  (\ref{4.6})-(\ref{4.8}) it follows that
\begin{eqnarray*}  \|uC_\varphi\|_{e,\mathcal{B}^\alpha \rightarrow \mathcal{Z}_\beta}&=&\inf\limits_{T\in K(\mathcal{B}^\alpha,\mathcal{Z}_\beta)}\|uC_\varphi-T\|_{\mathcal{B}^\alpha\rightarrow \mathcal{Z}_\beta}\nonumber\\ &\succeq & \max\Big\{\limsup\limits_{|\varphi(z_k)|\rightarrow1 } \frac{(1-|z_k|^2)^\beta |u''(z_k)|}{(1-|\varphi(z_k)|^2)^{\alpha-1}} ,  \nonumber\\ &&\limsup\limits_{|\varphi(z_k)|\rightarrow1 } \frac{(1-|z_k|^2)^\beta |2u'(z_k)\varphi'(z_k)+u(z_k)\varphi''(z_k)| }{(1-|\varphi(z_k)|^2)^\alpha},\nonumber\\ &&\limsup\limits_{|\varphi(z_k)|\rightarrow1 } \frac{(1-|z_k|^2)^\beta |u(z_k)(\varphi'(z_k))^2 |}{(1-|\varphi(z_k)|^2)^{\alpha+1}} \Big\}.  \end{eqnarray*}
Then using Lemma 2.1(ii) and Lemma 2.2(i) it follows the lower estimate. This completes the proof.
\end{proof}

\begin{thm} Weighted composition operator $uC_\varphi:\mathcal{B} \rightarrow \mathcal{Z}_\beta$ is bounded. Then  \begin{eqnarray*} \|uC_\varphi\|_{e,\mathcal{B} \rightarrow \mathcal{Z}_\beta}&\asymp&\max \Big\{\limsup\limits_{n\rightarrow \infty} (n+1) \|(2u'\varphi'+u\varphi'')\varphi^n\|_{\nu_\beta},\\ \nonumber && \;\;\;\;\;\;\;\;\;\;\;\;\;\;\;\limsup\limits_{n\rightarrow \infty} (n+1)^2\| u(\varphi')^2 \varphi^n\|_{\nu_\beta},\\ \nonumber && \;\;\;\;\;\;\;\;\;\;\;\;\;\;\;\;\;\limsup\limits_{n\rightarrow \infty}(\log n)\|u''\varphi^n\|_{\nu_\beta}.\Big\}.\end{eqnarray*}\end{thm}
\begin{proof} By Lemma 4.2 (ii), we obtain that \begin{eqnarray*} \|u''C_\varphi\|_{e,\mathcal{B}\rightarrow H_{\nu_\beta}^\infty} \asymp \limsup\limits_{n\rightarrow \infty} (\log n) \|u''\varphi ^n\|_{\nu_\beta}.\end{eqnarray*}
Then by  (\ref{4.2}) and the proof for the upper estimate in Theorem 4.3, we obtain the upper estimate.

The lower estimate. Similar to the proof in Theorem 4.4, we obtain (\ref{4.6})-(\ref{4.8}) with $\alpha=1$. Then \begin{eqnarray} \|uC_\varphi T\|_{e,\mathcal{B}^\alpha\rightarrow \mathcal{Z}_\beta} \succeq \max\Big\{\limsup\limits_{k\rightarrow \infty}  (1-|z_k|^2)^\beta |u''(z_k)|  ,  \nonumber\\ \limsup\limits_{k\rightarrow \infty} \frac{(1-|z_k|^2)^\beta |2u'(z_k)\varphi'(z_k)+u(z_k)\varphi''(z_k)| }{ 1-|\varphi(z_k)|^2 },\nonumber\\ \limsup\limits_{k\rightarrow \infty} \frac{(1-|z_k|^2)^\beta |u(z_k)(\varphi'(z_k))^2 |}{ (1-|\varphi(z_k)|^2)^2} \Big\}.  \label{4.9}\end{eqnarray}
By Lemma 2.1 (ii), Lemma 2.2 (i) and (\ref{4.9}) we have that  \begin{eqnarray*} \|uC_\varphi\|_{e,\mathcal{B} \rightarrow \mathcal{Z}_\beta}&\asymp&\max \Big\{\limsup\limits_{n\rightarrow \infty} (n+1) \|(2u'\varphi'+u\varphi'')\varphi^n\|_{\nu_\beta},\\ \nonumber && \;\;\;\;\;\;\;\;\;\;\;\;\;\;\;\limsup\limits_{n\rightarrow \infty} (n+1)^2\| u(\varphi')^2 \varphi^n\|_{\nu_\beta}.\Big\}.\end{eqnarray*}
Now choose the function  \begin{eqnarray*}\tilde{h}_k(z)=\frac{6}{a_k}\left(\log\frac{2}{1-\overline{\varphi(z_k)}z}\right)^2 -\frac{2}{a_k^2}\left( \log\frac{2}{1-\overline{\varphi(z_k)}z}\right)^3,\end{eqnarray*} where $a_k=\log\frac{2}{1-|\varphi(z_k)|^2}.$ It is easy to find that  the sequence $\{\tilde{h}_k\}\in\mathcal{B}$ converges to zero on the compact subset of $\mathbb{D}$, then  we can obtain \begin{eqnarray*}\limsup\limits_{k\rightarrow \infty}(1-|z_k|^2)^\beta |u''(z_k)|\log\frac{e}{1-|\varphi(z_k)|^2}\leq \|uC_\varphi\|_{e,\mathcal{B} \rightarrow \mathcal{Z}_\beta}.\end{eqnarray*} Further employing Lemma 2.1 (ii) and Lemma 2.2 (ii),  we get that \begin{eqnarray*} \|u C_\varphi\|_{e,\mathcal{B} \rightarrow \mathcal{Z}_\beta}\asymp\limsup\limits_{n\rightarrow \infty}(\log n)\|u''\varphi^n\|_{\nu_\beta}. \end{eqnarray*}
  This completes the proof.
\end{proof}

From the above Theorems, we obtain the new equivalent conditions  for the compactness of $uC_\varphi:\mathcal{B} \rightarrow \mathcal{Z}_\beta$.
\begin{cor} Suppose that $u\in H(\mathbb{D})$ and $\varphi\in S(\mathbb{D})$, then the followings are equivalent:

(a) $uC_\varphi:\mathcal{B}^\alpha \rightarrow \mathcal{Z}_\beta$ is compact.

(b) $uC_\varphi:\mathcal{B}^\alpha \rightarrow \mathcal{Z}_\beta$ is bounded and

(i) For $0<\alpha<1$, \begin{eqnarray}\limsup\limits_{n\rightarrow \infty} (n+1)^\alpha \|(2u'\varphi'+u\varphi'')\varphi^n\|_{\nu_\beta}=0,\label{4.10}\end{eqnarray}
\begin{eqnarray} \limsup\limits_{n\rightarrow \infty} (n+1)^{\alpha+1}\| u(\varphi')^2 \varphi^n\|_{\nu_\beta}=0. \label{4.11}\end{eqnarray}

(ii) For $\alpha=1$, $(\ref{4.10}), (\ref{4.11})$ hold and $\limsup\limits_{n\rightarrow \infty}(\log n)\|u''\varphi^n\|_{\nu_\beta}=0.$

(iii) For $\alpha>1$, $(\ref{4.10}), (\ref{4.11})$ hold and $\limsup\limits_{n\rightarrow \infty}(n+1)^{\alpha-1} \|u''\varphi^n\|_{\nu_\beta}=0.$
\end{cor}


\begin{thebibliography}{99}
\bibitem{CM} C.C. Cowen and B.D. MacCluer, Composition Operators on Spaces of Analytic Functions,
 CRC Press, Boca Raton, FL, 1995.


\bibitem{EL} K. Esmaeili and M. Lindstr\"{o}m, Weighted composition operators between Zygmund type spaces and their essential norms, Integr. Equ. Oper. Theory 75 (2013), 473-490.

\bibitem{FZ} Z.S. Fang and Z.H. Zhou, Essential norms of composition operators between Bloch type spaces in the polydisk, Arch. Math. 99(2012), 547-556.

\bibitem{GM} P. Gorkin and B. D. MacCluer, Essential norms of composition operators, Integr. Equ. Oper. Theory 48 (2004), 27-40.

\bibitem{HKLRS} O. Hyv\"{a}rinen, M. Kemappainen, M. Lindstr\"{o}m, A. Rautio and E. Saukko, The essential norms of weighted composition operators on weighted Banach spaces of analytic function, Integr. Equ. Oper. Theory 72 (2012), 151-157.

\bibitem{HL} O. Hyv\"{a}rinen and M. Lindstr\"{o}m, Estimates of essential norms of weighted composition operators between Bloch type spaces, J. Math. Anal. Appl. 393 (2012), 38-44.


\bibitem{LZ1} Y.X. Liang and Z.H. Zhou, New estimate of essential norm of composition followed by differentiation between Bloch-type spces, Banach J. Math. Anal. 8 (2014), 118-137.

\bibitem{LZ2} Y.X. Liang and Z.H. Zhou, Essential norm of product of  differentiation and composition operators between Bloch-type spaces, Arch. Math. 100 (4) (2013),  347--360.

\bibitem{M} A. Montes-Rodr\'{i}guez, Weighted composition operators on weighted Banach spaces of analytic functions, J. Lon. Math. Soc. 61 (2000), 872-884.

\bibitem{MZ} B. MacCluer and R. Zhao, Essential norms of weighted composition operators between Bloch-type spaces, Rocky Mountain J. Math. 33(2003),
1437-1458.

\bibitem{MZ1} J. Manhas and R. Zhao, New estimates of essential norms of weighted composition operators between Bloch type spaces, J. Math. Anal. Appl. 389 (2012), 32--47.

\bibitem{OSZ} S. Ohno, K. Stroethoff and R. Zhao, Weighted composition operators between Bloch-type spaces, Rocky Mountain J. Math. 33 (2003), 191-215.


\bibitem{Sha} J.H. Shapiro, Composition operators and classical function theory, Spriger-Verlag, 1993.

\bibitem{S} S. Stevi\'{c}, Essential norms of weighted composition operators from $\alpha$-Bloch space to a weighted-type space on the unit ball, Abstr. Appl. Anal. 2008 (2008), Art. ID 279691.

\bibitem{SCZ} S. Stevi\'{c}, R.Y. Chen and Z.H. Zhou, Weighted composition operators between Bloch-type spaces in the polydisc, Sbornik  Math. 201 (2010), 289-319.

\bibitem{WZZ} H. Wulan, D. Zheng and K. Zhu, Composition operators on BMOA and the Bloch space, Proc. Am. Math. Soc. 137 (2009), 3861-3868.


\bibitem{Zhao} R. Zhao,  Essential norms of composition operators between Bloch type spaces, Proc. Amer. Math. Soc. 138 (2010), 2537-2546.



\bibitem{Zhu3} K. Zhu, Bloch type spaces of analytic functions, Rocky Mountain J. Math. 23 (1993), 1143-1177.

\bibitem{ZS} Z.H. Zhou, J.H. Shi, Compactness of composition operators on the Bloch space in classical bounded symmetric domains, Michigan Math. J. 50 (2002), 381-405.

\bibitem{ZZ} H.G. Zeng and Z.H. Zhou, Essential norm estimate of a composition operator between Bloch-type spaces in the unit ball, Rocky Mountain J. Math. 42 (2012), 1049-1071.
\end{thebibliography}
\end{document}